\documentclass[12pt]{amsart}

\usepackage[margin = 1in]{geometry}
\usepackage{amsfonts, amsmath, amscd, amssymb,
latexsym, mathrsfs, mathtools, dsfont,
slashed, stmaryrd, verbatim, wasysym, bbm, url}
\usepackage[utf8]{inputenc}
\usepackage[backend=bibtex,bibstyle=authoryear,citestyle=alphabetic,giveninits=true,doi=false,isbn=false,url=false,backref]{biblatex}
\usepackage[usenames]{xcolor}
\usepackage{enumerate}
\usepackage[shortlabels]{enumitem}
\usepackage{scalerel}
      
\DeclareFieldFormat{labelalphawidth}{\mkbibbrackets{#1}}

\defbibenvironment{bibliography}
  {\list
     {\printtext[labelalphawidth]{%
        \printfield{labelprefix}%
    \printfield{labelalpha}%
        \printfield{extraalpha}}}
     {\setlength{\labelwidth}{\labelalphawidth}%
      \setlength{\leftmargin}{\labelwidth}%
      \setlength{\labelsep}{\biblabelsep}%
      \addtolength{\leftmargin}{\labelsep}%
      \setlength{\itemsep}{\bibitemsep}%
      \setlength{\parsep}{\bibparsep}}%
      }
  {\endlist}
  {\item}
\DeclareNameAlias{author}{last-first}
\newtheorem*{acknowledgements}{Acknowledgements}
\newtheorem*{theorem*}{Theorem}
\newtheorem{theorem}{Theorem}
\newtheorem{corollary}{Corollary}

\newtheorem{lemma}[corollary]{Lemma}
\newtheorem{proposition}[corollary]{Proposition}
\theoremstyle{definition}

\newtheorem{example}{Example}
\newtheorem{remark}[example]{Remark}

\numberwithin{equation}{section}
\numberwithin{theorem}{section}

\let\oldsqrt\sqrt
\def\sqrt{\mathpalette\DHLhksqrt}
\def\DHLhksqrt#1#2{%
\setbox0=\hbox{$#1\oldsqrt{#2\,}$}\dimen0=\ht0
\advance\dimen0-0.2\ht0
\setbox2=\hbox{\vrule height\ht0 depth -\dimen0}%
{\box0\lower0.4pt\box2}}

\usepackage{verbatim} 
\DeclareFontFamily{U}{mathx}{\hyphenchar\font45}
\DeclareFontShape{U}{mathx}{m}{n}{
      <5> <6> <7> <8> <9> <10>
      <10.95> <12> <14.4> <17.28> <20.74> <24.88>
      mathx10
      }{}
\DeclareSymbolFont{mathx}{U}{mathx}{m}{n}
\DeclareFontSubstitution{U}{mathx}{m}{n}
\DeclareMathAccent{\widecheck}{0}{mathx}{"71}

\newcommand\eps\varepsilon
\renewcommand\epsilon\varepsilon

\newcommand{\abs}[1]{\left\lvert #1 \right\rvert}
\newcommand{\smallabs}[1]{\lvert #1 \rvert}

\newcommand{\norm}[1]{\lVert #1 \rVert}
\newcommand\inner[1]{\langle #1 \rangle}

\newcommand\Mand{\text{ and }}

\newcommand\paperintro%
        {%
         }
\newcommand\paperbody%
        {%
         }

\newcommand\bbE{\mathbb{E}}

\newcommand\bbR{\mathbb{R}}

\newcommand\cC{\mathcal{C}}

\newcommand\cE{\mathcal{E}}

\newcommand\cI{\mathcal{I}}

\DeclareMathAlphabet{\mathpzc}{OT1}{pzc}{m}{it}

\newcommand{\sbs}{\subset}

\DeclareMathOperator{\Var}{Var}
\DeclareMathOperator{\Ent}{Ent}

\DeclareMathOperator{\Capa}{Cap}

\usepackage[refpage]{nomencl}

\makenomenclature

\usepackage{tikz, tikz-cd, tkz-euclide}
\usetikzlibrary{calc,positioning,intersections,quotes,decorations.markings}
\makeatletter
\def\@tocline#1#2#3#4#5#6#7{\relax
  \ifnum #1>\c@tocdepth 
  \else
    \par \addpenalty\@secpenalty\addvspace{#2}%
    \begingroup \hyphenpenalty\@M
    \@ifempty{#4}{%
      \@tempdima\csname r@tocindent\number#1\endcsname\relax
    }{%
      \@tempdima#4\relax
    }%
    \parindent\z@ \leftskip#3\relax \advance\leftskip\@tempdima\relax
    \rightskip\@pnumwidth plus4em \parfillskip-\@pnumwidth
    #5\leavevmode\hskip-\@tempdima
      \ifcase #1
       \or\or \hskip 1em \or \hskip 2em \else \hskip 3em \fi%
      #6\nobreak\relax
    \hfill\hbox to\@pnumwidth{\@tocpagenum{#7}}\par
    \nobreak
    \endgroup
  \fi}
\makeatother

\def\annu#1{_{%
  \vbox{\hrule height .2pt 
    \kern 1pt 
    \hbox{$\scriptstyle {#1}\kern 1pt$}%
  }\kern-.05pt 
  \vrule width .2pt 
}}

\allowdisplaybreaks
\usepackage[colorlinks,citecolor=blue,urlcolor=red,bookmarks=false,hypertexnames=true]{hyperref} 

\makeatletter
\def\keywords{\xdef\@thefnmark{}\@footnotetext}
\makeatother

\begin{document}

\title{$L^2$-stability for the variance Brascamp-Lieb inequality}

\author{K\'aroly J. B\"or\"oczky}
\address{Alfr\'ed R\'enyi Institute of Mathematics, Re\'altanoda utca 13-15, Budapest, Hungary, 1053, and E\"otv\"os University, Institute of Mathematics, P\'azmany P\'eter s\'et\'any 1/C, Budapest, Hungary, 1117}
\curraddr{}
\email{boroczky.karoly.j@renyi.hu}
\thanks{Supported by NKFIH grant 150613}

\author{Yaozhong Qiu}
\address{MODAL'X, Universit\'e Paris Nanterre, 92000 Nanterre, France}
\curraddr{}
\email{yqiu@parisnanterre.fr}
\thanks{This project has received funding from the European Union’s Horizon 2020 research and innovation programme under the Marie Sk\l{}odowska-Curie grant agreement No 101034255. \scalerel*{\includegraphics{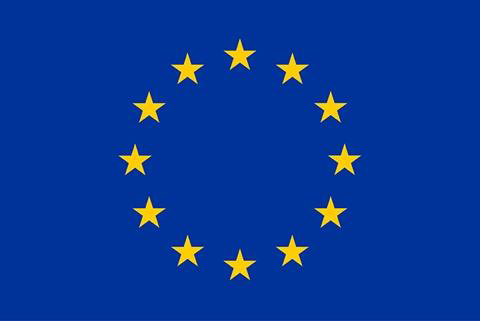}}{A}}

\author{Cyril Roberto}
\address{MODAL'X, Universit\'e Paris Nanterre, 92000 Nanterre, France}
\curraddr{}
\email{croberto@math.cnrs.fr}
\thanks{This work has been conducted within the FP2M federation (CNRS FR 2036) and FSM CY Initiative}


\date{}

\begin{abstract}
We prove an $L^2$-stability estimate for the variance Brascamp-Lieb inequality \cite{brascamp1976extensions} by bootstrapping the recent $L^1$-stability theorem of Machado and Ramos \cite{machado2025quantitative} under an additional assumption, which we call the super-Brascamp-Lieb inequality, of independent interest. 
\end{abstract}

\maketitle

\section{Introduction and main results}

In this paper, we prove $L^2$-stability estimates for the variance Brascamp-Lieb inequality \cite{brascamp1976extensions} which asserts for a log-concave probability measure of the form $d\mu_V = e^{-V}dx$ on $\bbR^n$ with $V \in \cC^2$ convex and satisfying $V'' > 0$ that
\begin{equation}\label{brascamp-lieb}    
    \Var_{\mu_V}(f) \leq \int_{\bbR^n} \inner{(V'')^{-1}f', f'}d\mu_V 
\end{equation}
for all $f: \bbR^n \rightarrow \bbR$ in $L^2(\mu_V)$. Here and in the sequel, if $u: \bbR^n \rightarrow \bbR$ is $\cC^2$, then $u' = \nabla u$ is its $n$-dimensional gradient, $u'' = D^2u$ its $n \times n$ Hessian matrix, $(V'')^{-1}$ the inverse matrix of $V''$, and $\inner{\cdot, \cdot}$ the standard scalar product of $\bbR^n$. 

Equality in \eqref{brascamp-lieb} is attained if and only if $f$ is of the form $f_{\sigma,\theta} \coloneq \sigma + \inner{\theta, V'}$ for some $\sigma \in \bbR$ and $\theta \in \bbR^n$. It is readily checked $\sigma = \int f_{\sigma, \theta}d\mu_V$ and since \eqref{brascamp-lieb} is invariant with respect to translation of $f$ to $f + c$ by any constant $c \in \bbR$, we may consider without loss of generality functions $f$ with mean zero $\int fd\mu_V = 0$ and extremisers of the form
\begin{equation}\label{extremisers}
    f_{\theta} \coloneq \inner{\theta, V'}, \quad \theta \in \bbR^n.
\end{equation}
The goal of this article is to establish the deficit 
\[ 
    \delta(f) = \int_{\bbR^n} \inner{(V'')^{-1}f', f'}d\mu_V  - \Var_{\mu_V}(f) 
\]
can be bounded below by some distance of $f$ to the space of extremisers. 

Before we state our main result, we first provide some review of the existing literature. To our best knowledge, the first stability result for the Brascamp-Lieb inequality is due to Harg\'e who proved \cite[Theorem~1]{harge2008reinforcement} $\delta(f) \geq C(n, V) (\int Vfd\mu_V)^2$. The deficit does not vanish \emph{only} on the space of extremisers as pointed out in \cite{cordero2017transport} but later Bolley, Gentil, and Guillin rediscovered Harg\'e's result with a different constant and showed in the same paper \cite[Theorem~4.3]{bolley2018dimensional} an estimate of the form $\delta(f) \geq \int \abs{f - \inner{f', (V'')^{-1}V'}}^2\omega(n, V)d\mu_V$ where $\omega(n, V)$ is a weight. Here the deficit does \emph{only} vanish for $f$ of the form \eqref{extremisers} but $\inner{f', (V'')^{-1}V'}$ is not of the form \eqref{extremisers} and a weight appears. This stands in contrast to the result of Cordero-Erausquin who proved \cite[Proposition~1.6]{cordero2017transport} 
\begin{equation}\label{stability-ce17}
    \delta(f) \geq C_2(V) \int \abs{f - f_{\theta^*}}^2d\mu_V 
    \Mand \delta(f) \geq C_1(V) \left(\int \abs{f - f_{\theta^*}}d\mu_V\right)^2
\end{equation} 
where $\theta^*$ is the barycentre
\[
    \theta^* = \theta^*(f) \coloneq \int xf(x)d\mu_V(x)
\]
of $f$ and $C_1(V), C_2(V) > 0$ are constants depending on $V$. For another $L^1$-stability estimate, see \cite[Proposition~4.3]{livshyts2024conjectural}, and for a refinement wherein $V''$ is replaced with a stronger matrix weight, we refer the reader to \cite{bonnefont2014intertwining, bonnefont2017intertwinings}.

A common feature of these results is that the stability depends on $V$. A natural question to ask is whether there exists a stability estimate with constants independent of $V$. This was recently positively answered by Machado and Ramos \cite[Theorem~1.1]{machado2025quantitative} who exploited the stability of the Pr\'ekopa-Leindler inequality due to Figalli, van Hintum, and Tiba \cite[Theorem~1.6]{figalli2025sharp} together with a proof of the Brascamp-Lieb inequality through the Pr\'ekopa-Leindler inequality by Bobkov and Ledoux \cite[Theorem~2.1]{bobkov2000brunn}. Specifically, the authors prove there exists $\theta = \theta(f) \in \bbR^n$ and a constant $C_1 = C_1(n)$ depending only on dimension $n$ such that 
\begin{equation}\label{l1-stability-1}
    \int \abs{f - f_{\theta}}d\mu_V \leq C_1 \sqrt{\delta(f)}. 
\end{equation}

The natural $L^2$-analogue of \eqref{l1-stability-1} would presumably follow from a conjectural $L^2$-stability estimate of the Pr\'ekopa-Leindler inequality which is not yet available in the literature. In this paper, we shall instead demonstrate a self-improvement argument which allows one to bootstrap the $L^1$-stability estimate into an $L^2$-estimate and in turn make explicit the choice of $\theta$ appearing in \eqref{l1-stability-1}. 

Our main results are the following. Suppose there exists $\delta \in (0,1)$ and $C_0 > 0$ such that for all $f$ smooth
\begin{equation}\label{eq:super-brascamp-lieb-intro}
    \Var_{\mu_V}(f) \leq \delta \int \inner{(V'')^{-1}f', f'}d\mu_V + C_0\left(\int \abs{f}d\mu_V\right)^2.
\end{equation}
Then there exists $\theta = \theta(f) \in \bbR^n$ 
and $C_2 = C_2(\delta, C_0, C_1)$ such that
\begin{equation}\label{l2-stability-1-intro}
    \int \abs{f - f_{\theta}}^2d\mu_V \leq C_2 \delta(f). 
\end{equation}
Moreover, for 
\[ 
    C_3 \coloneq C_1 + \sqrt{C_2\bbE_{\mu_V}(\abs{x}^2)}\bbE_{\mu_V}(\abs{V'}) \quad \Mand \quad C_4 \coloneq C_4(\delta, C_0, C_3), 
\]
one can replace $\theta$ by $\theta^*$ in \eqref{l1-stability-1} and in \eqref{l2-stability-1-intro} by substituting $C_1, C_2$ by $C_3, C_4$ respectively. See Theorem \ref{thm:l1+sbl=l2} below for a precise statement.

The main point is that \eqref{l2-stability-1-intro} asserts the $L^1$-stability estimate of Machado and Ramos can be improved into an $L^2$-estimate assuming \eqref{eq:super-brascamp-lieb-intro}. Moreover, although the extremiser $f_\theta$ in \eqref{l1-stability-1} is not explicit, since $\theta$ is deduced by compactness, one may expect $\theta$ can be taken as the barycentre $\theta^*$ as seen in Cordero-Erausquin's stability theorems \eqref{stability-ce17}. The second part of our result shows this is indeed the case and more precisely $L^2$-stability with respect to $\theta$ implies $L^1$-stability with respect to $\theta^*$ but we must pay a price, depending on $V$, in the constant. This latter $L^1$-estimate with explicit $\theta = \theta^*$ can in turn be again bootstrapped. We remark that this idea is quite general, in the sense that this $L^1$-to-$L^2$-stability improvement via \eqref{eq:super-brascamp-lieb-intro} also holds when starting from Cordero-Erausquin's $L^1$-stability estimate instead.  

In light of the qualitative similarity between \eqref{eq:super-brascamp-lieb-intro} and the super-log-Sobolev and super-Poincar\'e inequalities developed by Davies and Simon \cite{davies1984ultracontractivity} and Wang \cite{wang2000functional} respectively, we shall give \eqref{eq:super-brascamp-lieb-intro} the name \emph{super-Brascamp-Lieb inequality}, see \S3 for a precise definition. A secondary goal of this paper is to develop new technical tools that might help to understand the inequality as well as give examples of convex potentials for which \eqref{eq:super-brascamp-lieb-intro} is satisfied and thus the stability conclusions of Theorem \ref{thm:l1+sbl=l2} hold. 

The paper is organised as follows. In \S2, we prove our main stability theorem for the variance Brascamp-Lieb inequality assuming \eqref{eq:super-brascamp-lieb-intro}. In \S3, we study the super-Brascamp-Lieb inequality in detail, in particular related inequalities, tensorisation and glueing properties, one-dimensional Muckenhoupt-type conditions, and examples.

\section{$L^2$-stability for the variance Brascamp-Lieb inequality}

In this section we state and prove our main stability theorem for the variance Brascamp-Lieb inequality. Recall $f_{\theta} = \inner{\theta, V'}$ and $\theta^* = \theta^*(f) = \int xf(x)d\mu_V(x)$ from the introduction.

\begin{theorem}\label{thm:l1+sbl=l2}
Let $f: \bbR^n \rightarrow \bbR$ be smooth, such that  $\epsilon \coloneq \int \inner{(V'')^{-1} f',  f'}d\mu_V - \Var_{\mu_V}(f) > 0$ and $\int fd\mu_V = 0$. Suppose that for any $g: \bbR^n \rightarrow \bbR$ smooth
    \begin{equation}\label{eq:super-brascamp-lieb}
        \Var_{\mu_V}(g) \leq \delta \int \inner{(V'')^{-1} g',  g'}d\mu_V + C_0 \left(\int \abs{g}d\mu_V\right)^2
    \end{equation}
    for some $\delta \in (0, 1)$ and some $C_0 > 0$. Then there exists $\theta = \theta(f) \in \bbR^n$ such that
    \begin{equation}\label{l2-stability-1}
        \int \abs{f - f_{\theta}}^2d\mu_V \leq C_2\epsilon 
    \end{equation}
    with $C_2 \coloneq \frac{\delta + C_0C_1^2}{1 - \delta}$ where $C_1$ is defined in \eqref{l1-stability-1}.
    Furthermore, it holds
    \begin{equation}\label{l1-stability-2}
        \int \abs{f - f_{\theta^*}}d\mu_V \leq C_3\sqrt{\epsilon}
        \Mand 
        \int \abs{f - f_{\theta^*}}^2d\mu_V \leq C_4\epsilon
    \end{equation}
    with $C_3 \coloneq C_1 + \sqrt{C_2\bbE_{\mu_V}(\abs{x}^2)}\bbE_{\mu_V}(\abs{V'})$ and $C_4 \coloneq \frac{\delta + C_0C_3^2}{1-\delta}$.
\end{theorem}

Before we begin the proof of Theorem \ref{thm:l1+sbl=l2} let us first observe that $L^2$-stability is in fact equivalent to $H^1$-stability with respect to the norm induced by the Brascamp-Lieb energy $\int \inner{(V'')^{-1} f', f'}d\mu_V$.

\begin{lemma}\label{lem-l2-equals-h1}
    Assume $\int \inner{(V'')^{-1} f', f'}d\mu_V = \Var_{\mu_V}(f) + \epsilon$ with $\epsilon > 0$ and $\int fd\mu_V = 0$.
    Then for all $\theta \in \bbR^n$, it holds
    \[ 
        \int \inner{(V'')^{-1}(f' - f_\theta'), f' - f_\theta'}d\mu_V = \Var_{\mu_V} (f - f_\theta) + \epsilon. 
    \]
\end{lemma}

\begin{proof}
      We add and subtract $f_\theta$ from $f$ and expand both sides. On the left hand side we have 
    \begin{align*}
        \int \inner{(V'')^{-1} f', f'}d\mu_V 
        &= \int \inner{(V'')^{-1}(f' - f_\theta'), (f' - f_\theta')}d\mu_V + \int \inner{(V'')^{-1}f_\theta', f_\theta'}d\mu_V \\
        &+ 2\int \inner{(V'')^{-1}(f' - f_\theta'), f_\theta'}d\mu_V
    \end{align*}
    while on the right hand side we have 
    \[ 
        \Var_{\mu_V}(f) + \epsilon = \Var_{\mu_V}(f - f_\theta) + 2\int f_\theta(f - f_\theta)d\mu_V + \Var_{\mu_V}(f_\theta) + \epsilon. 
    \]
    Since $f_\theta$ is an optimiser, the terms $\int \inner{(V'')^{-1}f_\theta', f_\theta'}d\mu_V$ and $\Var_{\mu_V}(f_\theta)$ cancel, and because $f_\theta' = \inner{\theta, V'}' = V''\theta$, the cross term on the left 
    \[ 
        2 \int \inner{(V'')^{-1}(f' - f_\theta'), f_\theta'}d\mu_V = 2\int \inner{f' - f_\theta', \theta}d\mu_V = 2 \int \inner{\theta, V'}(f - f_\theta)d\mu_V 
    \]
    cancels with the cross term on the right. What remains is the claim. 
\end{proof}

\begin{proof}[Proof of Theorem \ref{thm:l1+sbl=l2}]
    Applying \eqref{eq:super-brascamp-lieb} to $g=f - f_\theta$ where $\theta$ is the one appearing in \eqref{l1-stability-1}, one has
    \begin{align*}
        \int \abs{f - f_\theta}^2d\mu_V 
        &\leq 
        \delta \int \inner
        {(V'')^{-1}(f' - f_\theta'), f' - f_\theta'}d\mu_V + C_0 \left(\int \abs{f - f_\theta}d\mu_V\right)^2 \\
        &\leq     
        \delta \int \inner
        {(V'')^{-1}(f' - f_\theta'), f' - f_\theta'}d\mu_V + C_0 C_1^2\epsilon \vphantom{\delta \left(\int \abs{f - f_\theta}^2d\mu_V + \epsilon\right) + C_0C_1^2\epsilon} \\
        &\leq 
        \delta \left(\int \abs{f - f_\theta}^2d\mu_V + \epsilon\right) + C_0C_1^2\epsilon 
    \end{align*}
    by \eqref{l1-stability-1} and Lemma \ref{lem-l2-equals-h1}, from which the the first part of the theorem follows.

    For the second part of the theorem, the question essentially boils down to a matter of whether the $\theta$ of \eqref{l1-stability-1} is $\sqrt{\epsilon}$-close to $\theta^*$. We observe that
    \[ 
        \int x \inner{\theta, V'}d\mu_V = -\int x \inner{\theta,(e^{-V})'}dx = \theta 
    \] 
    and hence 
    \[ 
        \abs{\theta^* - \theta}^2 
        = \abs{\int x(f - \inner{\theta, V'})d\mu_V}^2
        \leq \int \abs{f - \inner{\theta, V'}}^2d\mu_V \bbE_{\mu_V}(\abs{x}^2)
    \]
    so applying \eqref{l2-stability-1} yields $\abs{\theta^* - \theta}^2 \leq C_2\bbE_{\mu_V}(\abs{x}^2)\epsilon$. By the triangle inequality,
    \[ 
        \int \abs{f - \inner{\theta^* , V'}}d\mu_V 
        \leq 
        \int \abs{f - \inner{\theta,V'}}d\mu_V + \abs{\theta^* - \theta} \int \abs{V'}d\mu_V
    \] 
    from which we obtain the first part of \eqref{l1-stability-2} using the previous estimate on $\abs{\theta^* - \theta}$ and \eqref{l1-stability-1} to bound $\int \abs{f - \inner{\theta,  V'}}d\mu_V$. Repeating the first part of the proof with this new estimate yields the second part of \eqref{l1-stability-2}. 
\end{proof}

\section{The super-Brascamp-Lieb inequality: examples and related inequalities}\label{sec:sbl} 

In this section we study the super-Brascamp-Lieb inequality \eqref{eq:super-brascamp-lieb} and more generally the following one-parameter family of inequalities. We say that $\mu_V$ satisfies a \emph{super-Brascamp-Lieb inequality} if there exists $s_0 \geq 1$ and a function $\beta \colon [s_0, \infty) \to [0, \infty)$ such that, for any $s \geq s_0$ and any $g \colon \bbR^n \rightarrow \bbR$ smooth
\begin{equation} \label{eq:super-brascamp-lieb-full}
    \int g^2 d\mu_V \leq \beta(s) \int \inner{(V'')^{-1}g', g'}d\mu_V + s \left(\int \abs{g}d\mu_V\right)^2.
\end{equation}
The super-Brascamp-Lieb inequality is equipped with a function $\beta$ which we may assume without loss of generality, here and in the sequel, is non-increasing. The variance Brascamp-Lieb inequality implies one can always choose $s_0 = 1$ and $\beta \equiv 1$ but in light of Theorem \ref{thm:l1+sbl=l2} we seek $s > 1$ and $\beta(s) < 1$. 

\begin{remark}
    The conventions of \cite{davies1984ultracontractivity, wang2000functional} reverse the roles of $\beta$ and $s$; we adopt the convention of \cite{barthe2006interpolated} since we will follow their proofs more closely. 
\end{remark}

As a starting point, let us first observe that \eqref{eq:super-brascamp-lieb-full} immediately follows from the super-Poincar\'e inequality introduced and studied in \cite{wang2000functional} under the condition $V'' \leq \Lambda$. We say that $\mu_V$ satisfies a \emph{super-Poincar\'e inequality} if there exists a function $\beta: [1, \infty) \rightarrow [0, \infty)$, again non-increasing, such that for any $s \geq 1$ and any $f \colon \bbR^n \rightarrow \bbR$ smooth 
\begin{equation}\label{eq:super-poincare}
    \int f^2d\mu_V \leq \beta(s) \int \abs{f'}^2d\mu_V + s \left(\int \abs{f}d\mu_V\right)^2. 
\end{equation}
It is known that the log-Sobolev inequality implies \eqref{eq:super-poincare} with $\beta$ comparable to $1/\log s$ for $s$ large. Consequently, by the Bakry-\'Emery $\Gamma_2$-condition \cite{bakry1985diffusions}, it holds if $\Lambda \geq V'' \geq \rho > 0$ in which case the log-Sobolev inequality holds with constant $1/\rho$, see for instance \cite[\S5]{ane2000inegalites}. 

However, \eqref{eq:super-poincare} can hold for measures not satisfying the log-Sobolev inequality. Indeed, the standard examples satisfying \eqref{eq:super-poincare} are exponential power type measures $d\mu_p(x) \propto e^{-\abs{x}^p} dx$, $p > 1$, see \cite[\S7.7]{bakry2013analysis} for details. Moreover, for $p \in (1, 2]$ it was proved by \cite[Proposition~9]{barthe2007isoperimetry} that $\mu_p$ and their products $\mu_p^{\otimes n}$ satisfy \eqref{eq:super-poincare} with $\beta = \beta_p$ comparable to $C_p/(\log s)^{2(1-1/p)}$ for $s$ large and $C_p$ independent of $n$. 

To provide the reader intuition as to why one might expect validity of \eqref{eq:super-brascamp-lieb-full}, we note the double exponential measure $d\mu_1$ is a threshold in the sense it satisfies the Poincar\'e inequality but not \eqref{eq:super-poincare}. Thus it seems reasonable to expect there is a threshold log-concave measure satisfying the Brascamp-Lieb inequality \eqref{brascamp-lieb} but not \eqref{eq:super-brascamp-lieb-full} which may be associated with a threshold notion of convexity. In any case, we will prove in the following proposition and also in Proposition \ref{prop:bobkov-ledoux} the exponential power type measures $\mu_p$ satisfy \eqref{eq:super-brascamp-lieb-full} for $p > 1$ in a suitable sense, which might be contrasted with the fact that they also satisfy \eqref{eq:super-poincare}.

\begin{proposition}
\label{Proposition2}
    Assume $0 < V'' \leq \Lambda$ almost everywhere on $\bbR^n$ and assume $\mu_V$ satisfies the super-Poincar\'e inequality \eqref{eq:super-poincare} with $\beta = \beta_0$. Then $\mu_V$ satisfies the super-Brascamp-Lieb inequality \eqref{eq:super-brascamp-lieb-full} with $\beta = \Lambda \beta_0$. 
\end{proposition}

\begin{proof}
    If $\Lambda \geq V'' > 0$ then \eqref{eq:super-poincare} implies
    \[ \int f^2d\mu_V \leq \Lambda \beta(s) \int \inner{(V'')^{-1}f', f'}d\mu_V + s\left(\int \abs{f}d\mu_V\right)^2 \]
    which is \eqref{eq:super-brascamp-lieb-full}. 
\end{proof}

\begin{example}
    While the potentials $V(x) = \abs{x}^p$, $p \in (1, 2)$, defined on the line (for simplicity) do not satisfy $V'' \leq \Lambda$ at zero, they can be regularised according to \cite[\S3]{barthe2006interpolated} near zero to be $\cC^2$, convex, and bounded below by a positive constant depending on $p$. Moreover, the regularisation has tails like $e^{-\abs{x}^p}$ and so $V'' \rightarrow 0$ at infinity which implies $\Lambda \geq V''$.
\end{example}

\subsection{Additive $\varphi$-Brascamp-Lieb inequalities}

As for generalisations of the log-Sobolev inequality (see the introduction of \cite[\S5]{barthe2006interpolated} for an account), there are numerous ways to generalise the variance Brascamp-Lieb inequality. We shall focus on one specific inequality defined in analogy with the additive $\varphi$-Sobolev inequality which was introduced and studied in \cite{barthe2006interpolated}. 

We say that $\mu_V$ satisfies the \emph{additive $\varphi$-Brascamp-Lieb inequality} if there exists a finite constant $C_\varphi$ such that for any $g \colon \bbR^n \rightarrow \bbR$ smooth, 
\begin{equation}\label{eq:additive-brascamp-lieb}
  \int g^2 \varphi(g^2) d\mu_V
  - \int g^2 d\mu_V\varphi \left( \int g^2 d\mu_V \right) \leq C_\varphi \int \inner{(V'')^{-1}\nabla g, \nabla g}d\mu_V .
\end{equation}
For simplicity, we assume throughout the paper, here and in the sequel, that  $\varphi \colon (0, \infty) \rightarrow \bbR$ is strictly increasing and satisfies $\lim_{x \to \infty} \varphi(x) = \infty$. For technical reasons, we also assume $\lim_{x \to 0} x\varphi(x) = 0$ and $[0,\infty) \sbs \varphi((0,\infty))$. In particular, the inverse function $\varphi_+^{-1}: [0, \infty) \rightarrow (0, \infty)$ is increasing. 

The special case $\varphi = \log$ realising the entropy of $g^2$ on the left hand side of \eqref{eq:additive-brascamp-lieb} was considered by Bobkov and Ledoux \cite{bobkov2000brunn} who called the corresponding inequality the \emph{entropic Brascamp-Lieb inequality}. Furthermore, the authors gave a sufficient condition for its validity (see below). Other choices of $\varphi$ which have appeared in the literature in other contexts is the family $\varphi_\beta(x) = \log(1 + x)^\beta$, $\beta \in (0, 1]$.  

The main goal of this subsection is to prove the equivalence between super-Brascamp-Lieb and additive $\varphi$-Brascamp-Lieb inequalities, under mild assumptions on $\varphi$ and $\beta$. We anticipate that, modulo constants, heuristically $\varphi$ and $\beta$ are related through the relation $\varphi = 1/\beta$. 

However, we will encounter some technical difficulties since, for instance, $\varphi$ is defined on $(0, \infty)$ and $\beta$ on $[s_0, \infty)$ only, meaning we will need to carefully extend functions to the positive axis. Note also that the additive $\varphi$-Brascamp-Lieb inequality is tight, meaning it achieves equality for constant functions, while the super-Brascamp-Lieb inequality is not. 

\begin{proposition} \label{prop:equivalence}
The following holds. 
\begin{itemize}
    \item[(i)]
    Assume the additive $\varphi$-Brascamp-Lieb inequality \eqref{eq:additive-brascamp-lieb}. Then the super-Brascamp-Lieb inequality \eqref{eq:super-brascamp-lieb-full} holds with 
    $\beta(s) \coloneq \frac{4C_\varphi}{\varphi_+(s/4)}$ for $s \geq s_0$ with $s_0 = 4\varphi_+^{-1}(2D)$ where $D = \varphi(1) + \sup_{x: \varphi(x) < 0} x(-\varphi(x))$.
    \item[(ii)]
    Assume the super-Brascamp-Lieb inequality \eqref{eq:super-brascamp-lieb-full}. Let $\varphi(x) = \frac{1}{\beta(x)}$ for $x \geq s_0$ and suppose we may define $\varphi$ such that 
    \begin{enumerate}
        \item $\varphi$ is strictly increasing, $\cC^1$ on $(0, \infty)$, concave, and 
        \item there exists $\gamma > 0$ and $M \geq -\varphi(8)$ such that for all $x, y > 0$, one has $x \varphi'(x) \leq \gamma$ and $\varphi(xy) \leq M + \varphi(x) + \varphi(y)$. 
    \end{enumerate}
    Then the additive $\varphi$-Brascamp-Lieb inequality \eqref{eq:additive-brascamp-lieb} holds with constant 
    \[ C_\varphi = \frac{8}{(\sqrt{2} - 1)^2}(1 + \beta(s_0)(\varphi(8) + M)) + 2\gamma(1 + \sqrt{8s_0})^2. \] 
\end{itemize}
\end{proposition}


\begin{proof}
The proof borrows different elements from \cite{wang2000functional} and \cite{barthe2006interpolated, barthe2007isoperimetry}. 

We first prove Item $(i)$. Let $f$ be such that $\int \abs{f} d\mu_V = 1$. Assume $\abs{f} \geq \varepsilon$ for some $\varepsilon > 0$ and set $a = \int f^2d\mu$ and $b = \int \inner{(V'')^{-1}f', f'} d\mu_V$. Applying \eqref{eq:additive-brascamp-lieb} to $g^2 = f^2/\int f^2d\mu_V$ yields $\int f^2 \varphi(\frac{f^2}{a}) d\mu_V \leq C_\varphi b + \varphi(1)a$. Consequently, with $D$ given by the proposition, one has 
\begin{equation} \label{eq:start}
    \int f^2 \varphi_+\left(\frac{f^2}{a}\right) d\mu_V \leq C_\varphi b + Da. 
\end{equation}
Considering the two cases $t \leq \varphi_+(r^2/2)$ and $t \geq \varphi_+(r^2/2)$, it is readily checked
\[ 
    rt - t\sqrt{a\varphi_+^{-1}(t)} \leq r\varphi_+(r^2/a), \quad r \geq 0 \Mand t > 0. 
\] 
Inserting $r = \abs{f}$ and multiplying by $\abs{f}$, we obtain after integration
\[ 
    ta - t\sqrt{a\varphi_+^{-1}(t)} \leq \int f^2\varphi_+\left(\frac{f^2}{a}\right) d\mu_V, \quad t > 0.
\]
Hence, by \eqref{eq:start}, we have
\[
    (t - D)a - t \sqrt{a\varphi_+^{-1}(t)} - C_\varphi b \leq 0
\]
and therefore
\[
    \sqrt{a} \leq \frac{t\sqrt{\varphi_+^{-1}(t)}+\sqrt{t^2\varphi_+^{-1}(t)^2 + 4C_\varphi(t - D)b}}{2(t - D)}, \quad t > D.
\]
Squaring and since $(\sqrt{\alpha} + \sqrt{\alpha + \gamma})^2 \leq 4\alpha + 2\gamma$, we may remove the assumption $\int \abs{f}d\mu_V = 1$ by homogeneity and conclude
\[
    \int f^2 d\mu_V \leq \frac{2C_\varphi}{t-D} \int \inner{(V'')^{-1}f', f'} d\mu + \frac{t^2\varphi_+^{-1}(t)}{(t - D)^2} \left(\int \abs{f}d\mu_V\right)^2, \quad t > D.
\]
Restricting to $t \geq 2D$, we observe that
\[ 
    \frac{t^2\varphi_+^{-1}(t)}{(t - D)^2} \leq 4 \varphi_+^{-1}(t) \quad \Mand \quad \frac{2C_\varphi}{t - D} \leq \frac{4C_\varphi}{t}. 
\] 
Performing the change of variable $s = 4\varphi_+^{-1}(t)$, we obtain the desired conclusion of Item $(i)$ after removing the assumption $\abs{f} \geq \varepsilon$ by approximation. 

We now prove Item $(ii)$. Let $f$ be such that $\int f^2d\mu_V = 1$. For $n \geq 0$ set the dyadic level sets $A_n \coloneq \{2^n \leq f^2 \leq 2^{n+1}\}$ and define 
\[ 
    f_n \coloneq \min \left(1, \left(\frac{\abs{f} - \sqrt{2^{n-1}}}{\sqrt{2^n} - \sqrt{2^{n-1}}}\right)_+ \right)
\]
interpolating $0$ on $\{f^2 \leq 2^{n-1}\}$ and $1$ on $\{f^2 \geq 2^n\}$. By Markov's inequality,
\[ 
    \left(\int f_nd\mu_V\right)^2 \leq \mu_V\left\{f^2 \geq 2^{n-1}\right\} \int f_n^2d\mu_V \leq \frac{1}{2^{n-1}}\int f_n^2d\mu_V. 
\]
Hence, applying \eqref{eq:super-brascamp-lieb-full} to $f_n$ implies for $s \geq s_0$ that
\begin{align*}
    \int f_n^2d\mu_V 
    &\leq 
    \frac{\beta(s)}{(\sqrt{2^n} - \sqrt{2^{n-1}})^2} \cE_{V, n}(f) + s \left(\int f_n d\mu_V\right)^2 \\
    &\leq 
    \frac{\beta(s)}{2^{n-1}(\sqrt{2}-1)^2} \cE_{V, n}(f)  + \frac{s}{2^{n-1}} \int f_n^2 d\mu_V.
    \vphantom{\frac{\beta(s)}{(\sqrt{2^n} - \sqrt{2^{n-1}})^2} \cE_{V, n}(f) + s \left(\int f_n d\mu_V\right)^2}
\end{align*}
where $\cE_{V, n}(f) \coloneq \int_{A_n} \inner{(V'')^{-1}f', f'}d\mu_V \leq \int \inner{(V'')^{-1}f', f'}d\mu_V \eqqcolon \cE_V(f)$. Therefore, choosing $s = 2^{n-2}$, we have
\[ 
    \int f_n^2d\mu_V \leq \frac{\beta(2^{n-2})}{2^{n-2}(\sqrt{2}-1)^2} \cE_{V, n}(f). 
\] 
This is valid only if $2^{n-2} \geq s_0$. Let $n_0$ be such that $2^{n_0-3} < s_0 \leq 2^{n_0-2}$. Since $\int f_n^2 d\mu_V \geq \mu_V(f^2 \geq 2^n)$, we obtain
\begin{equation} \label{eq:n}
    2^{n+1} \mu_V(f^2 \geq 2^n) \leq \frac{8\beta(2^{n-2})}{(\sqrt{2}-1)^2} \cE_{V, n}(f), \quad n \geq n_0.   
\end{equation}
With $\varphi$ given by the proposition, set $\Phi_t(x) = x(\varphi(x) - \varphi(t)) - t\varphi'(t)(x - t)$ for $x, t > 0$. For $t = \norm{f}_2 = 1$, we observe that $\int f^2\varphi(f^2)d\mu_V - \int f^2d\mu_V\varphi(\int f^2d\mu_V)$ rewrites as
\[ 
    \int_{\{f^2 \leq 8s_0\}} \Phi_t(f^2)d\mu_V + \int_{\{f^2 > 8s_0\}} \Phi_t(f^2)d\mu_V. 
\]
For the second integral, since $\varphi$ is increasing by Assumption $(2)$, we find
\begin{align*}
    \int_{\{f^2 > 8s_0\}} \Phi_t(f^2)
    &\leq 
    \int_{\{f^2 > 8s_0\}} f^2\left( \varphi(f^2) - \varphi\left(\int f^2d\mu_V\right) \right) d\mu_V \\
    &\leq 
    \int_{\{f^2 > 8s_0\}} f^2\left( \varphi\left(\frac{f^2}{\int f^2 d\mu_V}\right) + M \right) d\mu_V \\
    &= \int f^2 F\left(\frac{f^2}{\int f^2 d\mu_V}\right) d\mu_V 
\end{align*}
where we have set $F(x) = \varphi(x) + M$ for $x > 8s_0$ and $F(x) = 0$ for $x \in [0, 8s_0]$. Using the sets $A_n$ introduced above, we have by \eqref{eq:n}
\begin{align*}
    \int f^2 F\left(\frac{f^2}{\int f^2 d\mu_V}\right) d\mu_V 
    &\leq 
    \sum_{n=n_0}^\infty \int_{A_n} f^2 F\left(\frac{f^2}{\int f^2 d\mu_V}\right) d\mu_V \vphantom{\frac{8}{(\sqrt{2} - 1)^2} \sum_{n = n_0}^\infty \beta(2^{n-2}) F(2^{n + 1}) \cE_{V, n}(f)} \\
    &\leq \sum_{n = n_0}^\infty 2^{n + 1}F(2^{n + 1}) \mu(f^2 \geq 2^n) \vphantom{\frac{8}{(\sqrt{2} - 1)^2} \sum_{n = n_0}^\infty \beta(2^{n-2}) F(2^{n + 1}) \cE_{V, n}(f)} \\
    &\leq \frac{8}{(\sqrt{2} - 1)^2} \sum_{n = n_0}^\infty \beta(2^{n-2}) F(2^{n + 1}) \cE_{V, n}(f).
\end{align*}
For all $x > s_0$, by Assumption $(2)$ applied to $y = 8$ one has
\begin{align*}
    \beta(x)F(8x) = \beta(x)(\varphi(8x) + M) 
    &\leq \beta(x)(\varphi(x) + \varphi(8) + 2M) \\
    &\leq 1 + \beta(s_0)(\varphi(8) + M) 
\end{align*} 
which, for $n \geq n_0$ and $x = 2^{n-2} \geq s_0$, implies that
\begin{align*}
    \int f^2 F\left(\frac{f^2}{\int f^2 d\mu_V}\right) d\mu_V 
    &\leq \frac{8}{(\sqrt{2} - 1)^2}(1 + \beta(s_0)(\varphi(8) + M)) \sum_{n = n_0}^\infty \cE_{V, n}(f) \\
    &\leq \frac{8}{(\sqrt{2} - 1)^2}(1 + \beta(s_0)(\varphi(8) + M)) \cE_V(f).
\end{align*}
For the first term $\int_{\{f^2 \leq 8s_0\}} \Phi_t(f^2)d\mu_V$, we observe that $\varphi(x^2) \leq \varphi(t^2) + \varphi'(t^2)(x^2 - t^2)$ by concavity of $\varphi$ so that
\[ 
    \Phi_t^2(x^2) \leq \varphi'(t^2)(x^2 - t^2)^2 = (x - t)^2\varphi'(t^2)(x + t)^2. 
\]
In particular, for $x = \abs{f}$ and $t = \norm{f}_2 = 1$, when $f^2 \leq 8s_0$, it holds
\begin{align*}
    \int_{\{f^2 \leq 8s_0\}} \Phi_t(f^2)
    &\leq 
    \int_{\{f^2 \leq 8s_0\}} (\abs{f}-t)^2 \varphi'(t^2) (1 + \sqrt{8s_0})^2 t^2 d\mu_V \vphantom{\gamma(1 + \sqrt{8s_0})^2 \int \left(\abs{f} - \norm{f}_2 \right)^2 d\mu_V} \\
    &\leq 
    \gamma(1 + \sqrt{8s_0})^2 \int \left(\abs{f} - \norm{f}_2 \right)^2 d\mu_V.
\end{align*}
Now,
\[ 
    \int \left(\abs{f} - \norm{f}_2 \right)^2 d\mu_V 
    = 2\left( \int f^2 d\mu_V - \int \abs{f}d\mu_V \norm{f}_2 \right). 
\]
Changing $f$ into $-f$ if necessary, we may assume without loss of generality that $\int fd\mu_V \geq 0$. By Cauchy-Schwarz, we deduce that $\left( \int fd\mu_V \right)^2 \leq \int \abs{f}d\mu_V \norm{f}_2$ so that, by the variance Brascamp-Lieb inequality \eqref{brascamp-lieb}
\[ 
    \int \left(\abs{f} - \norm{f}_2\right)^2d\mu_V \leq 2\Var_{\mu_V}(f) \leq 2\int \inner{(V'')^{-1}f', f'}d\mu_V. 
\]
Summing the two different estimates, we arrive at
\[ 
    \int f^2 \varphi(f^2)d\mu_V  - \int f^2 d\mu_V \varphi \left( \int f^2 d\mu_V \right) \leq C_\varphi \int \inner{(V'')^{-1}f', f'} d\mu_V 
\]
with $C_\varphi$ given by the proposition.
\end{proof}

\subsection{Tensorisation property}

Although the super-Brascamp-Lieb inequality does not tensorise independently of the dimension, the additive $\varphi$-Brascamp-Lieb inequality does, as we show. According to Proposition \ref{prop:equivalence}, the tensorisation property transfers to the super-Brascamp-Lieb inequality under mild assumptions on $\varphi, \beta$.

\begin{proposition} \label{prop:tensorisation}
Assume $\mu_{V_1}, \dots, \mu_{V_n}$ are probability measures on $\bbR$ satisfying the additive $\varphi$-Brascamp-Lieb inequality \eqref{eq:additive-brascamp-lieb} with functions $\varphi_i \equiv \varphi$ and constants $C_1, \dots, C_n$ respectively. Then the product measure $\mu_{V_1} \otimes \dots \otimes \mu_{V_n}$ on $\bbR^n$ satisfies \eqref{eq:additive-brascamp-lieb} with function $\varphi$ and constant $\max(C_1,\dots,C_n)$.
\end{proposition}

\begin{proof}
The proof is by induction and the inductive step is essentially the same as proving the base case $n = 2$. To simplify the presentation we only consider the product of two one-dimensional probability measures $\mu_{V_1}, \mu_{V_2}$. Let $V(x_1, x_2) = V_1(x_1) + V_2(x_2)$ be the potential underlying the product measure $\mu = \mu_{V_1} \otimes \mu_{V_2}$. Let $\Phi(x) = x\varphi(x)$ and set $\cE_{V, i}(g) = \int_\bbR (V_i'')^{-1}\abs{\partial_ig}^2d\mu_{V_i}$, $i = 1, 2$. Then
\begin{align*}
    \int_{\bbR^2} \Phi(g^2) d\mu_V
    &= \int_{\bbR} \left( \int_{\bbR} \Phi(g^2) d\mu_{V_1} \right) d\mu_{V_2} \\
    &\leq \int_{\bbR} \left( \Phi \left( \int_{\bbR} g^2 d\mu_{V_1}\right) + C_1 \cE_{V, 1}(g)\right) d\mu_{V_2}.
\end{align*}
Applying \eqref{eq:additive-brascamp-lieb} for $\mu_{V_2}$ to $G(x_2) = \left(\int_{\bbR}
g^2(x_1, x_2) d\mu_{V_1}(x_1) \right)^{1/2}$, we have
\[ 
    \int_\bbR \Phi \left( \int_{\bbR} g^2 d\mu_{V_1}\right) d\mu_{V_2} = \int \Phi(G^2)d\mu_{V_2} \leq \Phi\left(\int_\bbR G^2d\mu_{V_2}\right) + C_2 \cE_{V, 2}(G). 
\]
Since $\int_{\bbR} G^2 d\mu_{V_2} = \int_{\bbR^{2}} g^2 d(\mu_{V_1} \otimes \mu_{V_2})$ we obtain
\[ 
    \int_{\bbR^2} \Phi(g^2)d\mu_V - \Phi\left(\int_{\bbR^2} g^2d\mu_V\right) \leq \int_{\bbR} C_1 \cE_{V, 1}(g)d\mu_{V_2} + C_2 \cE_{V, 2}(G).
\]
The expected result follows from the Cauchy-Schwarz inequality
\[ 
    \abs{\partial_2 G}^2 = \frac{\abs{\int_\bbR g(x_1, x_2) \partial_2g(x_1, x_2)d\mu_{V_1}(x_1)}^2}{\int_\bbR g^2(x_1, x_2)d\mu_{V_1}(x_1)} \leq \int_\bbR \abs{\partial_2 g}^2 d\mu_{V_1}. \qedhere
\]
\end{proof}

\subsection{One-dimensional additive $\varphi$-Brascamp-Lieb inequalities}

Two examples of interest are the families of probability measures $\mu_p$ with density $d\mu_p(x) = Z_p^{-1}e^{-\abs{x}^p}dx$ on $\bbR$ and $\mu_p^+$ with density $d\mu_p^+(x) = (Z_p^+)^{-1}e^{-x^p} \mathds{1}_{[0, \infty)}(x)dx$ on the positive axis, where $Z_p^+, Z_p$ are normalisation constants and $p > 1$. The results of this subsection concerning the additive $\varphi$-Brascamp-Lieb transfer, thanks to Proposition \ref{prop:equivalence}, to the super-Brascamp-Lieb inequality. We will show results in two directions. We first improve an argument of Bobkov and Ledoux when $\varphi = \log$ and $\mu_V = \mu_p^+$, and then provide a general criterion \textit{\`a la} Muckenhoupt.

\subsubsection{Entropic Brascamp-Lieb inequalities for $\mu_p^+$}

Inequality \eqref{eq:additive-brascamp-lieb} for $\varphi = \log$ reads
\begin{equation} \label{eq:additive-brascamp-lieb2} 
    \Ent_{\mu_V}(g^2) = \int g^2 \log \left( \frac{g^2}{\int g^2d\mu_V} \right) d\mu_V \leq C \int \inner{(V'')^{-1}g', g'}d\mu_V.
\end{equation} 
In \cite{bobkov2000brunn} the authors show that, unlike the variance Brascamp-Lieb inequality, the entropic Brascamp-Lieb inequality may fail for some convex potentials $V$ for any constant. However, again by means of the Pr\'ekopa-Leindler inequality, the authors proved in \cite[Proposition~3.4]{bobkov2000brunn} that if for every $h \in \bbR^n$ the map $\bbR^n \ni x \mapsto \inner{V''(x)h, h}$ is concave, then the \eqref{eq:additive-brascamp-lieb2} holds with constant $C = 3$. This however is a very strong assumption; for instance, it does not hold for any $V$ supported on the line. It does hold on the positive axis for $\mu_p^+$ for $2 \leq p \leq 3$. In fact, specifying Bobkov and Ledoux's proof to the particular case $\mu_V = \mu_p^+$, it holds with constant $C = 2$ for all $p \geq 2$. This inequality can then be extended to the line and tensorised. To our best knowledge, this is the first instance of an $L^2$-stability estimate for the Brascamp-Lieb inequality where $V$ has unbounded Hessian, and hence not covered by Proposition~\ref{Proposition2}. 

\begin{proposition}\label{prop:bobkov-ledoux}
    Let $p \geq 2$ and $n \geq 1$. For all $f$ smooth, it holds
    \begin{equation}\label{eq:nu}
        \Ent_\nu(f^2) \leq C \int \inner{(V_{p, n}'')^{-1}f', f'} d\nu(x), \quad (V_{p, n})_{i,j} = \delta_{ij}p(p-1)\abs{x_j}^{p-2},
    \end{equation}
    where $C = 2$ if $\nu = (\mu_p^+)^{\otimes n}$ and $C = 3 + \log 2 \leq 4$ if $\nu = \mu_p^{\otimes n}$.
\end{proposition}

\begin{proof}
In both cases it suffices to consider $n = 1$ since the one-dimensional inequality can be tensorised according to Proposition \ref{prop:tensorisation}. 

For $\mu_p^+$, the proof of the entropic Brascamp-Lieb inequality of \cite[Proposition~3.4]{bobkov2000brunn} depends upon a lower bound of the form $L(s) \coloneq tV(x) + sV(y) - V(z) \gtrsim \inner{V''(z)k, k}$ where $x, y \in (0, \infty)$, $z = tx + sy$, $t + s = 1$, and $k = x - y$. Starting from the representation formula
\[ 
    L(s) = \frac{ts}{2}\int_0^1 (s\inner{V''(rz + (1-r)x)k, k} + t\inner{V''(rz + (1-r)y)k, k})dr^2, 
\]
one may prove $L(s) \geq \frac{ts}{3}\inner{V''(z)k, k}$ by using the convexity of $V$ and concavity of $V''$, the constant $3$ in the denominator manifesting in the constant of \cite[Proposition~3.4]{bobkov2000brunn}. Alternatively, returning to the representation formula, one can Taylor expand to higher order. Namely, 
\[ 
    V(x) = V(z) + V'(z)(x - z) + \frac{V''(z)}{2}(x - z)^2 + \frac{1}{2}\int_z^x V'''(r)(x - r)^2dr 
\] 
and similarly $V(y)$ satisfies the same formula with $y$ replacing $x$ everywhere. With $\cI(x) = \frac{1}{2}\int_z^x V'''(r)(x - r)^2dr$ the integral remainder above, we have 
\begin{align*}
    L(s) 
    &= \left[tV(z) + sV(z) - V(z) \right] + 
    \left[V'(z)(t(x - z) + s(y - z)) \right] \vphantom{\left[\frac{V''(z)}{2}(t(x - z)^2 + s(y - z)^2) \right] + 
    \left[ \frac{t}{2} \cI(x) + \frac{s}{2} \cI(y) \right]} \\ 
    &+ \left[\frac{V''(z)}{2}(t(x - z)^2 + s(y - z)^2) \right] + 
    \left[ \frac{t}{2} \cI(x) + \frac{s}{2} \cI(y) \right]
\end{align*}
and so the first and second terms vanish, the third term contains the bound we desire, and the fourth, if we assume $V''' \geq 0$, is nonnegative. This is valid for $\mu_p^+$, $p \geq 2$, and hence
\[ 
    L(s) \geq \frac{V''(z)}{2}(t(x - z)^2 + s(y - z)^2) = \frac{V''(z)}{2}(ts^2k^2 + s^2tk^2) = \frac{ts}{2} V''(z)k^2 
\]
since $x - z = -sk$ and $y - z = tk$. The proof then continues as before and we conclude \eqref{eq:nu} for $\mu_p^+$.

For $\mu_p$, let $f$ be smooth on $\bbR$ and denote by $\mu_p^-$ the measure on the negative axis with density $d\mu_p^-(x) = (Z_p^-)^{-1}e^{-(-x)^p} \mathds{1}_{(-\infty,0]}(x)dx$ and where $Z_p^- = Z_p^+$. Write $\mu_p = \frac{1}{2}(\mu_p^- + \mu_p^+)$. Then if $\int f^2 d\mu_p = 1$, one has 
\[ 
    2\Ent_{\mu_p}(f^2) 
    = 2 \int_\bbR f^2 \log f^2 d\mu_p
    = \int_{\bbR^-} f^2 \log f^2 d\mu_p^- + \int_{\bbR^+} f^2 \log f^2 d\mu_p^+.
\]
These two integrals contain the $\mu_p^\pm$-entropy up to a correction, namely 
\[ 
    \Ent_{\mu_p^-}(f^2) + \Ent_{\mu_p^+}(f^2) + \int f^2d\mu_p^- \log \left( \int f^2d\mu_p^- \right) + \int f^2d\mu_p^+ \log \left( \int f^2d\mu_p^+ \right)
\] 
which implies 
\[ 
    2\Ent_{\mu_p}(f^2) \leq \Ent_{\mu_p^-}(f^2) + \Ent_{\mu_p^+}(f^2) + 2\log 2
\]
since $\int f^2 d\mu_p^- + \int f^2 d\mu_p^+ = 2$ and $\sup_{x \in[0,2]} x \log x + (2-x)\log(2-x) = 2\log 2$. Since by symmetry \eqref{eq:nu} holds also for $\mu_p^-$, summing the two entropic Brascamp-Lieb inequalities on the negative and positive axes and removing by homogeneity the assumption $\int f^2d\mu_p = 1$ we arrive at  
\[ 
    \Ent_{\mu_p}(f^2) \leq \frac{1}{p(p-1)} \int_{-\infty}^\infty \abs{f'(x)}^2\abs{x}^{2-p}d\mu_p(x) + \log 2 \int f^2d\mu_p.
\]
By Rothaus' lemma \cite{rothaus1986hypercontractivity} asserting $\Ent_{\mu_p}(f^2) \leq 2 \Var_{\mu_p}(f) + \Ent_{\mu_p}(\widetilde{f}^2)$ for $\widetilde{f} = f - \int fd\mu_p$, we arrive at 
\begin{align*}
    \Ent_{\mu_p}(f^2) 
    &\leq 2 \Var_{\mu_p}(f) + \frac{1}{p(p-1)} \int_{-\infty}^\infty \smallabs{\widetilde{f}'(x)}^2\abs{x}^{2-p}d\mu_p(x) + \log 2 \int \widetilde{f}^2d\mu_p 
\end{align*}
and by applying the variance Brascamp-Lieb inequality \eqref{brascamp-lieb} on the first and third terms we conclude \eqref{eq:nu} for $\mu_p$.
\end{proof}

\subsubsection{Muckenhoupt criterion for additive $\varphi$-Brascamp-Lieb inequalities}

Originally developed for $L^p$-Hardy inequalities, $p \geq 1$, on the line with respect to two measures $\mu, \nu$, the Muckenhoupt criterion \cite{muckenhoupt1972hardy} characterises (up to universal constants) the Hardy constant in terms of a geometric constant defined in terms of the cumulative distribution functions of $\mu$ and $\nu$. Here we will develop a Muckenhoupt-type criterion for the additive $\varphi$-Brascamp-Lieb inequality. 

\begin{theorem}\label{thm:muckenhoupt}
Let $\mu_V$ be a probability measure on $\bbR$, with symmetric convex potential $V$. Let $\beta \colon [1,\infty) \to [0,\infty)$ be non-increasing. 
Define 
\[ 
    B = \sup_{x > 0} \frac{\mu_V((x,\infty))}{\beta(1/\mu_V((x,\infty)))} \int_0^x V''(t)e^{V(t)}dt \in [0,\infty].
\]
The following holds.
\begin{itemize}
    \item[(i)] Assume $s \mapsto s\beta(s)$ is non-decreasing on $[2,\infty)$ and that $B < \infty$. Then for all $f$ smooth and all $s \geq 1$,
    \[ 
        \int f^2 d\mu_V \leq 8 \beta(s) \int_\bbR \frac{{f'}^2}{V''} d\mu_V + s \left( \int \abs{f} d\mu_V \right)^2.
    \]
    \item[(ii)] Assume $\mu_V$ satisfies the super-Brascamp-Lieb inequality \eqref{eq:super-brascamp-lieb-full} and that there exists $\varphi$ as in Item $(ii)$ of Proposition  \ref{prop:equivalence}. Assume furthermore that $x \mapsto (\varphi(x)-\varphi(1))/x$ is
    non-increasing on $[2,\infty)$ and that there exists $\theta >4$ such that for all $x \geq 2,$ $\varphi(\theta x) - \varphi(1) \leq \frac{\theta}{4}(\varphi(x)-\varphi(1))$. Then $B < \infty$. 
\end{itemize}
\end{theorem}

The theorem shows that, under mild assumptions, the super-Brascamp-Lieb inequality holds if and only if the geometric constant $B$ is finite. For instance, Theorem \ref{thm:muckenhoupt} implies the entropic Brascamp-Lieb inequality (and by Proposition \ref{prop:equivalence} the super-Brascamp-Lieb inequality with $\beta(s) = 1 + \log s$, $s \geq 1$) holds if and only if
\[ 
    B_{\log} = \sup_{x > 0} \mu_V((x,\infty)) \left(1 + \log \left( \frac{1}{\mu_V((x,\infty))} \right)\right) \int_0^x V''(t)e^{V(t)}dt < \infty. 
\]

The proof of Theorem \ref{thm:muckenhoupt} is based on  general ideas developed in \cite{barthe2003sobolev, barthe2006interpolated, barthe2007isoperimetry}, after the seminal work of Bobkov and G\"otze \cite{bobkov1999exponential} who first established the connection between Sobolev-type inequalities in probability spaces in dimension 1 and Muc\-ken\-houpt-type criteria.

\begin{proof}
We provide a sketch for brevity. Following the proof of \cite[Theorem 1]{barthe2007isoperimetry} we find
\[ 
    \int f^2 d\mu_V - s \left( \int \abs{f} d\mu_V \right)^2 \leq 4 B_s \int \frac{{f'}^2}{V''}d\mu_V
\]
for any $s \geq 1$, with
$B_s$ the smallest constant so that for all $A \subset (0,\infty)$ 
\[ 
    B_s \Capa(A, (0, \infty)) \geq \frac{\mu_V(A)}{1+(s-1)\mu_V(A)} .
\]
Here $\Capa(A, (0, \infty))$ denotes the capacity of a set $A$ with respect to $(0,\infty)$, a notion introduced in \cite{barthe2003sobolev} and adapting the well-known (electrostatic) capacity of Maz'ya \cite{mazya1985prostranstva} to probability spaces. The only difference with \cite[Theorem 1]{barthe2007isoperimetry} is that the energy term in the super-Brascamp-Lieb inequality comes with a measure different from $\mu_V$, say $\nu$ with density $d\nu(x) = (V'')^{-1}e^{-V}dx$, that contains the ``Brascamp-Lieb" weight $(V'')^{-1}$.

The key observation is that (see the proof of \cite[Theorem 2.3]{barthe2005concentration} and \cite[Appendix]{barthe2003sobolev} for details) 
\[
    \Capa(A, (0, \infty)) = \Capa((\inf A, \infty), (0, \infty)) = \frac{1}{\int_0^{\inf A} V''(t)e^{V(t)}dt}.
\]
Since 
\[ \frac{\mu_V(A)}{1+(s-1)\mu_V(A)} \leq \frac{\mu_V((\inf A,\infty))}{1+(s-1)\mu_V((\inf A,\infty))}, \]
the smallest constant $B_s$ above reduces to  
\[ 
    B_s = \sup_{x >0} \frac{\mu_V((x,\infty))}{1 + (s-1)\mu_V((x,\infty))} \int_0^x V''(t)e^{V(t)}dt 
\]
Finally, under the assumption of the theorem, it is proved in \cite[Corollary 6]{barthe2007isoperimetry} that 
\[ 
    \sup_{s \geq 1} \frac{\mu_V((x,\infty))}{[1+(s-1)\mu_V((x,\infty))]\beta(s)} \leq \frac{2\mu_V((x,\infty))}{\beta(1/\mu_V((x,\infty))}
\]
from which Item $(i)$ follows.

For Item $(ii)$ we observe that \eqref{eq:super-brascamp-lieb-full} implies \eqref{eq:additive-brascamp-lieb} by Proposition \ref{prop:equivalence}. Then we invoke \cite[Theorem~22]{barthe2006interpolated}, see also \cite[\S5.6]{barthe2006interpolated}, together with the relation between the capacity and $\int_0^x V''(t)e^{V(t)}dt$, to conclude that $B < \infty$.
\end{proof}

\begin{acknowledgements}
    \textup{We warmly thank the referee for their thorough reading of the paper and for their feedback which helped improve the exposition of the paper.}  
\end{acknowledgements}

\printbibliography

@book{ane2000inegalites,
  title={Sur les in{\'e}galit{\'e}s de Sobolev logarithmiques},
  author={An{\'e}, C{\'e}cile and Blach{\`e}re, S{\'e}bastien and Chafa{\"\i}, Djalil and Foug{\`e}res, Pierre and Gentil, Ivan and Malrieu, Florent and Roberto, Cyril and Scheffer, Gr{\'e}gory},
  volume={10},
  year={2000},
  publisher={Soci{\'e}t{\'e} math{\'e}matique de France Paris}
}

@incollection{bakry1985diffusions,
  title={Diffusions hypercontractives},
  author={Bakry, Dominique and {\'E}mery, Michel},
  booktitle={S{\'e}minaire de Probabilit{\'e}s XIX 1983/84: Proceedings},
  pages={177--206},
  year={1985},
  publisher={Springer}
}

@book{bakry2013analysis,
  title={Analysis and Geometry of Markov diffusion operators},
  author={Bakry, Dominique and Gentil, Ivan and Ledoux, Michel},
  volume={348},
  year={2013},
  publisher={Springer Science \& Business Media}
}

@article{barthe2003sobolev,
  title={Sobolev inequalities for probability measures on the real line},
  author={Barthe, Franck and Roberto, Cyril},
  journal={Stud. Math.},
  volume={159},
  number={3},
  pages={481--497},
  year={2003}
}

@article{barthe2005concentration,
 author = {Barthe, Franck and Cattiaux, Patrick and Roberto, Cyril},
 title = {Concentration for independent random variables with heavy tails},
 fjournal = {AMRX. Applied Mathematics Research eXpress},
 journal = {AMRX, Appl. Math. Res. Express},
 issn = {1687-1200},
 volume = {2005},
 number = {2},
 pages = {39--60},
 year = {2005},
 doi = {10.1155/AMRX.2005.39},
 zbMATH = {5017596},
 Zbl = {1094.60010}
}

@article{barthe2006interpolated,
 author = {Barthe, Franck and Cattiaux, Patrick and Roberto, Cyril},
 title = {Interpolated inequalities between exponential and {Gaussian}, {Orlicz} hypercontractivity and isoperimetry},
 fjournal = {Revista Matem{\'a}tica Iberoamericana},
 journal = {Rev. Mat. Iberoam.},
 issn = {0213-2230},
 volume = {22},
 number = {3},
 pages = {993--1067},
 year = {2006},
 doi = {10.4171/RMI/482},
 url = {https://eudml.org/doc/42001},
 zbMATH = {5149127},
 Zbl = {1118.26014}
}

@article{barthe2007isoperimetry,
 author = {Barthe, Franck and Cattiaux, Patrick and Roberto, Cyril},
 title = {Isoperimetry between exponential and {Gaussian}},
 fjournal = {Electronic Journal of Probability},
 journal = {Electron. J. Probab.},
 issn = {1083-6489},
 volume = {12},
 pages = {1212--1237},
 year = {2007},
 doi = {10.1214/EJP.v12-441},
 url = {https://eudml.org/doc/128514},
 zbMATH = {5214066},
 Zbl = {1132.26005}
}

@article{bobkov2000brunn,
  title={From Brunn-Minkowski to Brascamp-Lieb and to logarithmic Sobolev inequalities},
  author={Bobkov, Sergey G and Ledoux, Michel},
  journal={Geom. Funct. Anal.},
  volume={10},
  number={5},
  pages={1028--1052},
  year={2000},
  publisher={Springer}
}

@article{bobkov1999exponential,
 author = {Bobkov, S. G. and G{\"o}tze, F.},
 title = {Exponential integrability and transportation cost related to logarithmic {Sobolev} inequalities},
 fjournal = {Journal of Functional Analysis},
 journal = {J. Funct. Anal.},
 issn = {0022-1236},
 volume = {163},
 number = {1},
 pages = {1--28},
 year = {1999},
 doi = {10.1006/jfan.1998.3326},
 zbMATH = {1278792},
 Zbl = {0924.46027}
}

@article{bolley2018dimensional,
  title={Dimensional improvements of the logarithmic Sobolev, Talagrand and Brascamp--Lieb inequalities},
  author={Bolley, Fran{\c{c}}ois and Gentil, Ivan and Guillin, Arnaud},
  journal={Ann. Probab.},
  volume={46},
  number={1},
  pages={261--301},
  year={2018},
  publisher={JSTOR}
}

@article{bonnefont2014intertwining,
 author = {Bonnefont, Michel and Joulin, Ald{\'e}ric},
 title = {Intertwining relations for one-dimension-al diffusions and application to functional inequalities},
 fjournal = {Potential Analysis},
 journal = {Potential Anal.},
 issn = {0926-2601},
 volume = {41},
 number = {4},
 pages = {1005--1031},
 year = {2014},
 doi = {10.1007/s11118-014-9408-7},
 zbMATH = {6369341},
 Zbl = {1311.60086}
}

@article{bonnefont2017intertwinings,
 author = {Arnaudon, Marc and Bonnefont, Michel and Joulin, Ald{\'e}ric},
 title = {Intertwinings and generalized {Brascamp}-{Lieb} inequalities},
 fjournal = {Revista Matem{\'a}tica Iberoamericana},
 journal = {Rev. Mat. Iberoam.},
 issn = {0213-2230},
 volume = {34},
 number = {3},
 pages = {1021--1054},
 year = {2018},
 doi = {10.4171/RMI/1014},
 zbMATH = {6966782},
 Zbl = {1428.60112}
}

@article{brascamp1976extensions,
 author = {Brascamp, Herm Jan and Lieb, Elliott H.},
 title = {On extensions of the {Brunn}-Minkows-ki and {Prekopa}-{Leindler} theorems, including inequalities for log concave functions, and with an application to the diffusion equation},
 fjournal = {Journal of Functional Analysis},
 journal = {J. Funct. Anal.},
 issn = {0022-1236},
 volume = {22},
 pages = {366--389},
 year = {1976},
 doi = {10.1016/0022-1236(76)90004-5},
 zbMATH = {3522267},
 Zbl = {0334.26009}
}

@article{cordero2017transport,
  title={Transport inequalities for log-concave measures, quantitative forms, and applications},
  author={Cordero-Erausquin, Dario},
  journal={Canad. J. Math.},
  volume={69},
  number={3},
  pages={481--501},
  year={2017},
  publisher={Cambridge University Press}
}

@article{davies1984ultracontractivity,
  title={Ultracontractivity and the heat kernel for Schr{\"o}d\-inger operators and Dirichlet Laplacians},
  author={Davies, Edward Brian and Simon, Barry},
  fjournal={Journal of Functional Analysis},
  journal={J. Funct. Anal.},
  volume={59},
  number={2},
  pages={335--395},
  year={1984},
  publisher={Elsevier}
}

@article{figalli2025sharp,
  title={Sharp Quantitative Stability for the Pr\'ekopa-Leindler and Borell-Brascamp-Lieb Inequalities},
  author={Figalli, Alessio and {van Hintum}, Peter and Tiba, Marius},
  journal={arXiv:2501.04656},
  year={2025}
}

@article{harge2008reinforcement,
 author = {Harg{\'e}, Gilles},
 title = {Reinforcement of an inequality due to {Brascamp} and {Lieb}},
 fjournal = {Journal of Functional Analysis},
 journal = {J. Funct. Anal.},
 issn = {0022-1236},
 volume = {254},
 number = {2},
 pages = {267--300},
 year = {2008},
 doi = {10.1016/j.jfa.2007.07.019},
 zbMATH = {5234461},
 Zbl = {1138.28002}
}

@article{livshyts2024conjectural,
 author = {Livshyts, Galyna V.},
 title = {On a conjectural symmetric version of {Ehrhard}'s inequality},
 fjournal = {Transactions of the American Mathematical Society},
 journal = {Trans. Am. Math. Soc.},
 issn = {0002-9947},
 volume = {377},
 number = {7},
 pages = {5027--5085},
 year = {2024},
 doi = {10.1090/tran/9177},
 zbMATH = {7891376},
 Zbl = {1544.52006}
}

@article{machado2025quantitative,
  title={Quantitative stability for the Brascamp-Lieb inequality and moment measures},
  author={Machado, Jo{\~a}o Miguel and Ramos, Jo{\~a}o P. G.},
  journal={arXiv:2511.22636},
  year={2025}
}

@book{mazya1985prostranstva,
 author = {Maz'ya, V. G.},
 title = {Prostranstva {S}. {L}. {Soboleva}},
 year = {1985},
 publisher = {Leningrad: Izdatel'stvo Leningradskogo Universiteta},
 language = {Russian},
 zbMATH = {44836},
 Zbl = {0727.46017}
}

@article{muckenhoupt1972hardy,
 author = {Muckenhoupt, Benjamin},
 title = {Hardy's inequality with weights},
 fjournal = {Studia Mathematica},
 journal = {Stud. Math.},
 issn = {0039-3223},
 volume = {44},
 pages = {31--38},
 year = {1972},
 doi = {10.4064/sm-44-1-31-38},
 url = {https://eudml.org/doc/217718},
 zbMATH = {3374075},
 Zbl = {0236.26015}
}

@article{rothaus1986hypercontractivity,
 author = {Rothaus, O. S.},
 title = {Hypercontractivity and the {Bakry}-{Emery} criterion for compact {Lie} groups},
 fjournal = {Journal of Functional Analysis},
 journal = {J. Funct. Anal.},
 issn = {0022-1236},
 volume = {65},
 pages = {358--367},
 year = {1986},
 doi = {10.1016/0022-1236(86)90025-X},
 zbMATH = {3947286},
 Zbl = {0589.58036}
}

@article{wang2000functional,
  title={Functional inequalities for empty essential spectrum},
  author={Wang, Feng Yu},
  journal={J. Funct. Anal.},
  volume={170},
  number={1},
  pages={219--245},
  year={2000},
  publisher={Elsevier}
}

\end{document}